\documentclass[
reqno]{amsart}
\usepackage{latexsym,amsmath,amssymb,amscd}
\usepackage[all]{xy}

\makeatletter
  \newcommand\@dotsep{4.5}
  \def\@tocline#1#2#3#4#5#6#7{\relax
     \ifnum #1>\c@tocdepth 
     \else
     \par \addpenalty\@secpenalty\addvspace{#2}%
     \begingroup \hyphenpenalty\@M
     \@ifempty{#4}{%
     \@tempdima\csname r@tocindent\number#1\endcsname\relax
        }{%
         \@tempdima#4\relax
           }%
      \parindent\z@ \leftskip#3\relax \advance\leftskip\@tempdima\relax
      \rightskip\@pnumwidth plus1em \parfillskip-\@pnumwidth
       #5\leavevmode\hskip-\@tempdima #6\relax
       \leaders\hbox{$\m@th
       \mkern \@dotsep mu\hbox{.}\mkern \@dotsep mu$}\hfill
       \hbox to\@pnumwidth{\@tocpagenum{#7}}\par
       \nobreak
        \endgroup
         \fi}
\makeatother 
\begin{document}


\makeatletter
\@addtoreset{figure}{section}
\def\thefigure{\thesection.\@arabic\c@figure}
\def\fps@figure{h,t}
\@addtoreset{table}{bsection}

\def\thetable{\thesection.\@arabic\c@table}
\def\fps@table{h, t}
\@addtoreset{equation}{section}
\def\theequation{
\arabic{equation}}
\makeatother

\newcommand{\bfi}{\bfseries\itshape}

\newtheorem{theorem}{Theorem}
\newtheorem{corollary}[theorem]{Corollary}
\newtheorem{definition}[theorem]{Definition}
\newtheorem{example}[theorem]{Example}
\newtheorem{lemma}[theorem]{Lemma}
\newtheorem{notation}[theorem]{Notation}
\newtheorem{problem}[theorem]{Problem}
\newtheorem{proposition}[theorem]{Proposition}
\newtheorem{remark}[theorem]{Remark}

\numberwithin{theorem}{section}
\numberwithin{equation}{section}

\renewcommand{\1}{{\bf 1}}

\newcommand{\Ad}{{\rm Ad}}
\newcommand{\ad}{{\rm ad}}
\newcommand{\Ci}{{\mathcal C}^\infty}

\newcommand{\Aut}{{\rm Aut}}
\newcommand{\ph}{\text{\bf P}}
\newcommand{\de}{{\rm d}}
\newcommand{\con}{\leadsto}
\newcommand{\diag}{{\rm diag}}
\newcommand{\ee}{{\rm e}}
\newcommand{\GL}{{\rm GL}}
\newcommand{\ie}{{\rm i}}
\newcommand{\id}{{\rm id}}
\renewcommand{\Im}{{\rm Im}}
\newcommand{\Ker}{{\rm Ker}\,}
\newcommand{\Lie}{\textbf{L}}
\newcommand{\Op}{{\rm Op}}
\newcommand{\Ran}{{\rm Ran}\,}
\newcommand{\rank}{{\rm rank}\,}
\newcommand{\spann}{{\rm span}}
\renewcommand{\sl}{{\mathfrak s}{\mathfrak l}}
\newcommand{\so}{{\mathfrak s}{\mathfrak o}}
\newcommand{\su}{{\mathfrak s}{\mathfrak u}}

\newcommand{\Tr}{{\rm Tr}\,}

\newcommand{\CC}{{\mathbb C}}
\newcommand{\HH}{{\mathbb H}}
\newcommand{\NN}{{\mathbb N}}
\newcommand{\RR}{{\mathbb R}}

\newcommand{\Ac}{{\mathcal A}}
\newcommand{\Bc}{{\mathcal B}}
\newcommand{\Cc}{{\mathcal C}}
\newcommand{\Dc}{{\mathcal D}}
\newcommand{\Hc}{{\mathcal H}}
\newcommand{\Jc}{{\mathcal J}}
\newcommand{\Kc}{{\mathcal K}}
\renewcommand{\Mc}{{\mathcal M}}
\newcommand{\Sc}{{\mathcal S}}
\newcommand{\Xc}{{\mathcal X}}
\newcommand{\Yc}{{\mathcal Y}}
\newcommand{\Zc}{{\mathcal Z}}
\newcommand{\Ag}{{\mathfrak A}}

\newcommand{\ag}{{\mathfrak a}}
\newcommand{\fg}{{\mathfrak f}}
\renewcommand{\gg}{{\mathfrak g}}
\newcommand{\hg}{{\mathfrak h}}
\renewcommand{\ng}{{\mathfrak n}}
\newcommand{\Sg}{{\mathfrak S}}
\newcommand{\zg}{{\mathfrak z}}


\title[Contractions of Lie algebras with 2-dimensional coadjoint orbits]
{Contractions of Lie algebras with 2-dimensional generic coadjoint orbits}
\author{Daniel Belti\c t\u a}
\author{Benjamin Cahen}
\date{\today}
\address{Institute of Mathematics 
of the Romanian Academy,
P.O. Box 1-764, Bucharest, Romania}
\email{beltita@gmail.com, Daniel.Beltita@imar.ro}
\address{Laboratoire de Math\'ematiques et Applications de Metz, UMR 7122, 
Universit\'e de Lorraine (campus de Metz) et CNRS,
B\^at. A, Ile du Saulcy, F-57045 Metz Cedex 1, France}
\email{benjamin.cahen@univ-lorraine.fr}
\thanks{The research of the first-named author has been partially supported by the Grant
of the Romanian National Authority for Scientific Research, CNCS-UEFISCDI,
project number PN-II-ID-PCE-2011-3-0131.}
\keywords{contraction; Lie algebra; coadjoint orbit}
\subjclass[2010]{Primary 17B05; Secondary 17B08, 17B81}

\begin{abstract}
We determine all the contractions within the class of finite-dim\-ension\-al real Lie algebras 
whose coadjoint orbits have dimensions~$\le2$.
\end{abstract}


\maketitle


\section{Introduction}

The notion of contraction of Lie algebras was introduced on physical grounds by Segal \cite{Se51}, In\"on\"u  and Wigner
\cite{IW53}: 
If two physical theories are related by a limiting process then the associated invariance groups should also be related
by some limiting process called contraction. 
For instance, classical mechanics is a limiting case of relativistic mechanics and then the Galilei group
is a contraction of the Poincar\'e group.

Contractions of Lie algebras have been investigated by many authors \cite{Sa61},
\cite{He66}, \cite{LeN67}, \cite{We91} and continue to be a subject of active interest, particularly
in connection with the somewhat inverse problem of deforming Lie algebras \cite{FM05},
\cite{Bu07}. 
Note that contractions not only link two Lie algebras but also link some objects related to these Lie algebras
such as representations, invariants, special functions and quantization mappings \cite{MN72},
\cite{DR85}, \cite{CW99},\cite{Cp07}, 
\cite{Ca09}, 
and also coadjoint orbits, which provide the motivation for the present paper, 
as we will explain directly, below.

The coadjoint orbits of any Lie group $G$ admit $G$-invariant symplectic structures, 
and may be regarded as phase spaces acted on by the group $G$ in a Hamiltonian fashion, 
in the sense of classical mechanics. 
That well-known observation allows us to regard the Lie groups with 
2-dimensional coadjoint orbits 
as symmetry groups of the simplest nontrivial phase spaces, in some sense. 
Therefore it is natural to wonder which ones of these symmetry groups 
can be further contracted. 
That is precisely the question which we answer in the present paper 
(see Theorem~\ref{main}), 
by considering contractions on Lie algebra level (Definitions \ref{contr0_def}~and~\ref{contr_def}). 

Note that contractions of any of the aforementioned symmetry groups of the simplest nontrivial phase spaces 
necessarily belong to the same class of simplest symmetry groups.  
More precisely, for any finite-dimensional real Lie algebra~$\gg$ associated with the simply connected Lie group~$G$, 
the maximal dimension of the coadjoint orbits of $G$ is an isomorphism invariant that does not increase 
for any contraction of $\gg$ (see Lemma~\ref{NP} below). 
Thus, 
our results can also be regarded as a contribution to understanding 
the contraction relationships within particular classes of Lie algebras. 
Here are some samples of problems that were earlier raised on such relationships 
within various classes of Lie algebras: 
\begin{itemize}
\item Which are the real Lie algebras that do not admit other contractions 
than the abelian Lie algebras and themselves? 
This was answered in \cite{La03}. 
\item Which are all the contractions for low-dimensional Lie algebras? 
This was settled in \cite{NP06} for the Lie algebras of dimensions~$\le 4$. 
\item Is it true that within the class of nilpotent Lie algebras, 
every algebra is a nontrivial contraction of another algebra? 
This is the Grunewald-O'Halloran conjecture (\cite{GO93}) 
which was recently addressed in \cite{HT13}. 
\end{itemize}
It is noteworthy that the class of Lie algebras investigated by us 
(the ones with 2-dimensional generic coadjoint orbits) 
is restricted neither by dimension nor by nilpotency conditions. 

As indicated above, the present investigation was motivated by a question 
that claims its origins in classical physics. 
Therefore the next stage of our research will naturally focus 
on answering the similar question on the quantum level, 
that is, by studying the contractions of the unitary representations 
obtained by the quantization of the simplest phase spaces. 
In other words, from the mathematical point of view of the method of coadjoint orbits, 
our results raise several interesting problems related to 
the contractions of unitary irreducible representations of the groups 
with coadjoint orbits of dimensions~$\le 2$. 
That is a broad topic which was already treated in \cite{Ca03} and \cite{Ca04} 
for some particularly important situations, 
and we plan to address it systematically in future papers.

The present paper is organized as follows. 
In Section~\ref{results} we state the contraction problem to be addressed 
and we also fix some terminology and state our main results 
as Theorem~\ref{main}, thereby providing a complete answer to that problem. 
Section~\ref{Sect3} then collects some auxiliary facts on contractions. 
Section~\ref{Sect3} has a rather technical character 
and are devoted to the proof of Theorem~\ref{main}(\eqref{main_item1}--\eqref{main_item2}). 
Finally, the proof of that theorem is completed in Section~\ref{Sect4}, 
which also includes
some additional observations that may be of independent interest.


\section{Main results}\label{results}

Here is  the main problem to be addressed below. 
In its statement we also introduce some notation to be used throughout the present paper.  

\begin{problem}\label{probl1}
\normalfont
One determined in \cite{ACL86} the list of all of the Lie algebras corresponding to 
the connected, simply connected Lie groups whose coadjoint orbits have the dimensions~$\le2$, namely: 
\begin{itemize}
\item[(i)] the simple Lie algebras $\su(2)$ and $\sl(2,\RR)$; 
\item[(ii)] the solvable Lie algebra with a 1-codimensional abelian ideal $\gg_T:=\RR\ltimes_T \ag_n$, 
where $T\colon\ag_n\to\ag_n$ is the linear operator defined by the adjoint action 
of $(1,0)\in\RR\ltimes\ag_n$; 
\item[(iii)] the solvable Lie algebra $\RR T\ltimes\hg_3$, 
where $\hg_3=\spann\{X,Y,Z\}$ is the Heisenberg algebra with $[X,Y]=Z$, 
and\footnote{See the notation $A_{4,8}^{-1}$ and $A_{4,9}^0$ 
in the list of 4-dimensional algebras in \cite[\S VI.B]{NP06}.}
\begin{itemize}
\item either $[T,X]=X$, $[T,Y]=-Y$, $[T,Z]=0$ \hfill(the Lie algebra $A_{4,8}^{-1}$); 
\item or $[T,X]=Y$, $[T,Y]=-X$, $[T,Z]=0$ \hfill(the Lie algebra $A_{4,9}^0$); 
\end{itemize}
\item[(iv)] the 2-step nilpotent Lie algebra  
$\ng_{3,3}=\spann\{X_1,X_2,X_3,Y_1,Y_2,Y_3\}$ with $[X_1,X_2]=Y_1$, $[X_2,X_3]=Y_2$, $[X_3,X_1]=Y_3$ 
(the free 2-step nilpotent Lie algebra of rank~3); 
\item[(v)] the 3-step nilpotent Lie algebra $\ng_{2,1,2}=\spann\{X_1,X_2,X_3,Y_1,Y_2\}$ with $[X_1,X_2]=X_3$, $[X_1,X_3]=Y_1$, $[X_2,X_3]=Y_2$;  
\end{itemize}
and moreover any direct sums of the above Lie algebras with abelian Lie algebras. 

Which ones of these Lie algebras are contractions of other Lie algebras from the above list?  
\end{problem}

Let us recall the definition of a contraction of Lie algebras.

\begin{definition}\label{contr0_def}
\normalfont 
Let $\gg$ and $\gg_0$ be finite-dimensional real Lie algebras. 
We say that $\gg_0$ is a \emph{contraction} of the Lie algebra $\gg$, and we write $\gg\con\gg_0$, 
if there exists a family of invertible linear maps $\{C_r\colon\gg_0\to\gg\}_{r\in I}$ parameterized 
by the set $I\subseteq\RR$ for which $0\in\RR$ is an accumulation point and such that 
for all $x,y\in\gg_0$ we have 
$[x,y]_{\gg_0}=\lim\limits_{I\ni r\to 0}C_r^{-1}[C_rx,C_ry]_{\gg}$. 
\end{definition}

In order to describe our answer to the above problem, 
it is convenient to introduce a notion of contraction between Lie algebras 
which may not have the same dimension. 
This notion is well defined and recovers the classical notion of contraction 
in the case of Lie algebras having the same dimension, 
as proved in Proposition~\ref{18sept2013} below. 

\begin{definition}\label{contr_def}
\normalfont 
Let $\gg$ and $\gg_0$ be finite-dimensional real Lie algebras. 
We say that $\gg_0$ is a \emph{stabilized contraction} of the Lie algebra $\gg$, and we write $\gg\con_s\gg_0$, 
if there exist some integers $k,k_0\ge 0$ with $k+\dim\gg=k_0+\dim\gg_0$ 
for which there exists a family of invertible linear maps $\{C_r\colon\gg_0\times\ag_k\to\gg\times\ag_{k_0}\}_{r\in I}$ parameterized 
by the set $I\subseteq\RR$ for which $0\in\RR$ is an accumulation point and such that 
for all $x,y\in\gg_0\times\ag_{k_0}$ we have 
$[x,y]_{\gg_0\times\ag_{k_0}}=\lim\limits_{I\ni r\to 0}C_r^{-1}[C_rx,C_ry]_{\gg\times\ag_k}$. 
\end{definition}

We can now summarize our main results as the following theorem, 
whose statement Lie algebras that occur in the above Problem~\ref{probl1}, 
as well as a few other Lie algebras introduced in Notation~\ref{algebras} below. 

\begin{theorem}\label{main}
Here are all the contraction relationships that exist among the Lie algebras of the types (i)--(v) 
from Problem~\ref{probl1}.  In Statement (1) we refer to stabilized contractions and in Statement (2) and Statement (3),
we refer to usual contractions.
\begin{enumerate}
\item\label{main_item1} 
Among the Lie algebras of types (i) and (iii)--(v) we have: 
$$
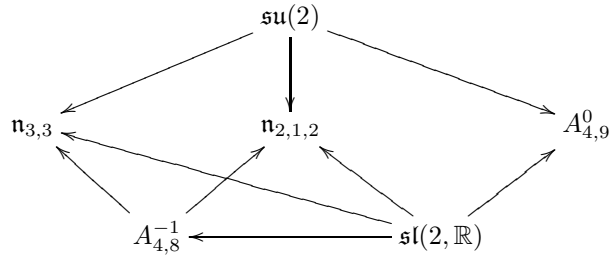
\xymatrix{
 &  & \su(2) \ar[dll] \ar[d] \ar[drr]& & \\
 \ng_{3,3}&  & \ng_{2,1,2}& & A_{4,9}^0 \\
 &A_{4,8}^{-1} \ar[ul] \ar[ur]  & &\sl(2,\RR) \ar[ulll] \ar[ul] \ar[ur] \ar[ll]&  
}$$
\item\label{main_item2} 
Any of the Lie algebras of types (i) and (iii)--(v) contracts to Lie algebras of type (ii) as follows: 
\begin{itemize}
\item $\su(2)\con\gg_T\iff\gg_T\in\{\ag_3, \hg_3, A_{3,5}^0\}$
\item $\su(2)\times\RR\con\gg_T\iff\gg_T\in\{\ag_4, \hg_3\times\RR, A_{3,5}^0\times\RR, \fg_4, A_{4,9}^0\}$ 
\item $\sl(2,\RR)\con\gg_T\iff\gg_T\in\{\ag_3, \hg_3, A_{3,4}^{-1}, A_{3,5}^0\}$ 
\item $\sl(2,\RR)\times\RR\con\gg_T\Leftrightarrow\gg_T\in\{\ag_4, \hg_3\times\RR, A_{3,4}^{-1}\times\RR, 
A_{3,5}^0\times\RR, \fg_4, A_{4,9}^0\}$
\item $A_{4,8}^{-1}\con\gg_T 
\iff \gg_T\in\{\ag_4,\ag_1 \times A_{3,4}^{-1}\}$
\item $A_{4,9}^0 \con\gg_T 
\iff \gg_T\in\{\ag_4,\ag_1 \times A_{3,5}^0\}$ 
\item $\ng_{3,3} \con\gg_T 
\iff \gg_T\in\{\ag_6,\ag_3\times \hg_3,\ag_1\times\ng_{1,2,2}\}$
\item $\ng_{2,1,2} \con\gg_T 
\Leftrightarrow \gg_T\in\{\ag_5,\ag_2\times \hg_3,\ng_{1,2,2},\ag_1\times\fg_4\}$
\end{itemize}
\item\label{main_item3} 
The Lie algebras of type (ii) contract to each other as follows: 
$${\mathfrak g}_T\con {\mathfrak g}_S\iff S\in \overline {C(T)\setminus (0)}$$ 
where for any square matrix $T$ 
we denote by $C(T)$ is the set of all nonzero scalar multiples of the matrices in the similarity orbit of $T$, 
while the overline stands for the topological closure. 
\end{enumerate}
\end{theorem}

The proof of Assertions \eqref{main_item1}--\eqref{main_item2} of the above theorem will be given in Section~\ref{Sect4}, 
and Assertion~\eqref{main_item3} will be proved in Section~\ref{Case5}.  

Here is the list of the Lie algebras that occur in the statement of Theorem~\ref{main} 
and which were not introduced in Problem~\ref{probl1}: 

\begin{notation}\label{algebras}
\normalfont
We use the following notation for $n\ge1$: 
\begin{itemize}
\item $\hg_{2n+1}$ is the $(2n+1)$-dimensional Heisenberg algebra, 
which can be described as the Lie algebra with a basis $X_1,\dots,X_n,Y_1,\dots,Y_n,Z$ 
and the Lie bracket defined by $[X_j,Y_j]=Z$ for $j=1,\dots,n$. 
\item $\fg_{n+2}$ is the $(n+2)$-dimensional filiform Lie algebra, 
which is the Lie algebra with a basis $X_0,\dots,X_n,Y$ and the Lie bracket defined by 
$[Y,X_j]=X_{j-1}$ for $j=1,\dots,n$. 
\item $\ag_n$ is the $n$-dimensional abelian Lie algebra. 
We also define $\ag_0=\{0\}$.  
\end{itemize}
Moreover we use the following Lie algebras: 
\begin{itemize}
\item  $A_{3,5}^0$ is the 3-dimensional Lie algebra (cf. \cite[\S VI.A]{NP06}) 
defined by the commutation relations 
$[e_1,e_3]=-e_2$, $[e_2,e_3]=e_1$.
\item $\ng_{1,2,2}$ is the 5-dimen\-sional 2-step nilpotent Lie algebra 
defined by the commutation relations 
$[e_1,e_2]=e_4$, $[e_1,e_3]=e_5$.
\end{itemize}
\end{notation}

\section{Preliminaries on contractions of Lie algebras}\label{Sect3}

\begin{notation}
\normalfont 
For every finite-dimensional real Lie algebra $\gg$ we will use the following notation: 
\begin{itemize}
\item the Lie bracket $[\cdot,\cdot]_{\gg}\colon\gg\times\gg\to\gg$; 
\item the dual space $\gg^*:=\{\xi\colon\gg\to\RR\mid\xi\text{ is linear}\}$; 
\item the duality pairing 
$\langle\cdot,\cdot\rangle\colon\gg^*\times\gg\to\RR$
defined by $\langle\xi,x\rangle=\xi(x)$ for all $\xi\in\gg^*$ and $x\in\gg$; 
\item for every $\xi\in\gg^*$ we define 
$$B^{\gg}_\xi\colon\gg\times\gg\to\RR, \quad B^{\gg}_\xi(x,y)=\langle\xi,[x,y]_{\gg}\rangle=\langle(\ad^*_{\gg}y)\xi,x\rangle$$ 
where $\ad^*_{\gg}\colon\gg\times\gg^*\to\gg^*$ is the infinitesimal coadjoint action;
\item $\rank(\ad^*_{\gg}):=\max\{\rank  B^{\gg}_\xi\mid\xi\in\gg^*\}$.  
\end{itemize}
\end{notation}

\begin{remark}
\normalfont 
If $G$ is any Lie group whose Lie algebra is $\gg$, then $\rank(\ad^*_{\gg})$ 
is the maximum of the dimensions of the coadjoint orbits of~$G$. 
\end{remark}

\begin{lemma}\label{NP}
If  $\gg\con\gg_0$, 
then $\rank(\ad^*_{\gg_0})\le\rank(\ad^*_{\gg})$. 
\end{lemma}

\begin{proof}
See \cite[Th. 1(10)]{NP06}. 
\end{proof}

\begin{remark}\label{aux}
\normalfont
We record here a few properties of the above families of Lie algebras: 
\begin{enumerate}
\item\label{aux_item1} We have $\hg_3=\fg_3$. 
\item\label{aux_item2} It was proved in \cite[Prop. 5]{ACMP83} that $\rank(\ad^*_{\fg_m})=2$ for every $m\ge 3$. 
\item\label{aux_item3} We have $\rank(\ad^*_{\hg_{2n+1}})=2n$ for every $n\ge 1$. 
\item\label{aux_item4} It follows by \cite[Th. 5.2]{La03} (see also \cite[Th. 1(i)]{Go91}) 
that if $\gg$ is an $m$-dimensional nilpotent Lie algebra 
such that $\gg\con\gg_0$ if and only if $\gg_0=\ag_m$, then 
$\gg=\hg_3\times\ag_{m-3}$. 
\end{enumerate}
\end{remark}

\begin{proposition}\label{prop1}
Let $m\ge n\ge 3$ be any positive integers and assume that $n$ is odd. 
Then we have $\fg_m\con\hg_n\times\ag_{m-n}$ if and only if $n=3$. 
\end{proposition}

\begin{proof}
We have $\rank(\ad^*_{\hg_n\times\ag_{m-n}})=\rank(\ad^*_{\hg_n})=n-1$ and $\rank(\ad^*_{\fg_m})=2$ 
by Remark~\ref{aux}(\eqref{aux_item2}--\eqref{aux_item3}). 
Therefore, if $\fg_m\con\hg_n\times\ag_{m-n}$, then Lemma~\ref{NP} entails $n-1\le 2$, hence necessarily $n=3$. 

For the converse assertion, let $X_0,\dots,X_{m-1},Y$ be a basis of $\fg_m$ as in Notation~\ref{algebras}. 
For fixed $a_0,\dots,a_{m-1}\in\RR$ 
and for every $r>0$ define a linear map $C_r\colon\fg_m\to\fg_m$, by $C_r(Y)=Y$, $C_r(X_j)=r^{a_j}Xj$ if $0\le j\le m-1$. 
Then we have $C_r^{-1}[C_r(Y),C_r(X_j)]=r^{a_j}C_r^{-1}(X_{j-1})=r^{a_j-a_{j-1}}X_{j-1}$ for $j=1,\dots,m-1$. 
Hence if $a_{m-1}>\cdots>a_1=a_0$ then we obtain 
$$\lim\limits_{r\to 0}C_r^{-1}[C_r(Y),C_r(X_j)]=
\begin{cases}
0 &\text{ if }2\le j\le m-1,\\
X_0 &\text{ if }j=0.
\end{cases}$$
Since $\RR Y+\RR X_1+\RR X_0\simeq \hg_3$, we thus see that $\fg_m\con\hg_3\times\ag_{m-3}$. 
\end{proof}

As a related result we note the following rigidity property of the nilpotent Lie groups with square-integrable representations modulo the center. 

\begin{proposition}
Let $G$ and $G_0$ be nilpotent Lie groups with 1-dimensional centers and with the Lie algebras $\gg$ and $\gg_0$. 
If $G_0$ has square-integrable representations modulo the center and $\gg\con\gg_0$, 
then also $G$ has square integrable representations modulo the center. 
\end{proposition}

\begin{proof} 
Let $n=\dim\gg$. 
Since the center of $G$ is 1-dimensional, 
it has square integrable representations modulo the center if and only if $\rank(\ad^*_{\gg})=n-1$ 
(see \cite{MW73}). 
For the same reason we have $\rank(\ad^*_{\gg_0})=n-1$, 
and this also shows that the integer $n$ is odd, since the dimension of the coadjoint orbits is always an even integer. 
On the other hand 
Lemma~\ref{NP} ensures that $\rank(\ad^*_{\gg})\ge\rank(\ad^*_{\gg_0})=n-1$,  
hence the conclusion follows since $\rank(\ad^*_{\gg})$ is an even integer less than~$n$. 
\end{proof}

\begin{proposition}\label{18sept2013}
The property described in Definition~\ref{contr_def} does not depend on the choice of the integers $k,k_0\ge 0$. 
\end{proposition}

\begin{proof}
An easy reasoning by induction shows that it suffices to prove the following: 
If $\gg$ and $\gg_0$ are finite-dimensional real Lie algebras 
with $\dim\gg=\dim\gg_0$ and there exists a family of invertible linear maps 
$\{C_r\colon\gg_0\times\ag_1\to\gg_1\times\ag_1\}_{r\in I}$ parameterized 
by the set $I\subseteq\RR$ for which $0\in\RR$ is an accumulation point and 
\begin{equation}\label{18sept2013_proof_eq1}
(\forall x,y\in\gg_0\times\ag_1)\quad  
[x,y]_{\gg_0\times\ag_1}=\lim\limits_{I\ni r\to 0}C_r^{-1}[C_rx,C_ry]_{\gg\times\ag_1}
\end{equation}
then  there also exists a family of invertible linear maps 
$\{A_r\colon\gg_0\to\gg_1\}_{r\in I}$ for which
\begin{equation}\label{18sept2013_proof_eq2}
(\forall x,y\in\gg_0)\quad  
[x,y]_{\gg_0}=\lim\limits_{I\ni r\to 0}A_r^{-1}[A_rx,A_ry]_{\gg}.
\end{equation}
To this end let us write 
$C_r=\begin{pmatrix}
A_r & b_r \\
c_r & d_r 
\end{pmatrix}$
where $A_r\colon\gg_0\to\gg$, $b_r\in\gg$, $c_r\in\gg_0^*$, and $d_r\in\RR$. 

If $c_r=0$, then the invertibility property of $C_r$ 
entails that $A_r$ is invertible and moreover $A_r$ is the restriction of $C_r$ to $\gg_0$. 
Therefore, if $c_r=0$ for $r\in I$ close enough to $0$, then 
\eqref{18sept2013_proof_eq1} implies \eqref{18sept2013_proof_eq2}. 

We now show how the general case can be reduced to the situation that we just discussed. 
First note that  if $A_r$ is invertible then we may define an automorphism of $\gg\times\ag_1$ 
by 
$F_r:=\begin{pmatrix}
\id & 0 \\
-c_rA_r^{-1} & 1
\end{pmatrix}$, 
hence 
we may replace $C_r$ by $F_rC_r$ in order to assume that also $c_r=0$. 

Finally, by replacing a general contraction $C_r$ by a suitable perturbation 
$C_r+\epsilon_r\id$ for some $\varepsilon_r\in\RR$, the general case can be reduced 
to the case of a contraction whose component $A_r$ is invertible, 
and then the above discussion applies.   
\end{proof}

\section{Proof of Theorem~\ref{main}(\eqref{main_item1}--\eqref{main_item2})}\label{Sect4}

In order to prove the theorem, we take into account the 13 possible situations, 
each of them having some subcases: 

\subsection{(i) vs. (i)}\label{Case0}  
\begin{itemize}
\item[(a)] Does $\RR^k\times\su(2)\con\RR^k\times \sl(2,\RR)$ hold true? 
 No, since 
$\RR^k\times\su(2)\not\simeq\RR^k\times \sl(2,\RR)$ while the algebras of derivations 
of these two Lie algebras have the same dimension $3+k^2$, hence we may use \cite[Th. 1(1)]{NP06}. 
\item[(b)] Does $\RR^k\times\sl(2,\RR)\con\RR^k\times \su(2)$ hold true? 
 No, for the same reason as above. 
Alternatively, since the Killing form of $\RR^k\times \su(2)$ has 3 negative eigenvalues 
while the Killing form of $\RR^k\times\sl(2,\RR)$ has only one negative eigenvalue, 
and the number of negative eigenvalues of the Killing form cannot increase by a contraction process; 
see \cite[Th. 1(16)]{NP06}. 
\end{itemize}

\subsection{(i) vs. (ii)}\label{Case1}  
\begin{itemize}
\item[(a)] Does $\RR^{n-2}\times$(i)$\con$(ii) hold true, if $n\ge 2$?  
Since the dimension of the derived algebra cannot increase by a contraction 
(\cite[Th. 1(4)]{NP06}) it follows by \eqref{derived} that necessarily $\dim(\Ran T)\le 3$, 
and then we may assume $n\le 3$. 
All the possible contractions of  $\su(2)$, $\sl(2,\RR)$ (our case $n=2$), 
$\su(2)\times\RR$, $\sl(2,\RR)\times\RR$ (our case $n=3$) were 
determined in \cite[pag. 26--27]{NP06}. 
There are 2 possible situations: 
\\ ({\bfi a1}) The situation involving  $\su(2)$:  
\\ $\bullet$ \fbox{$\su(2)\con\gg_T\iff\gg_T\in\{\ag_3, \hg_3, A_{3,5}^0\}$}
where $A_{3,5}^0$ is defined on $\RR^3$ by the commutation relations \eqref{A350}. 
If we define $\su(2)$ by the commutation relations 
$[e_1,e_2]=e_3,\ [e_2,e_3]=e_1,\ [e_3,e_1]=e_2$
then a contraction $\su(2)\con A_{3,5}^0$ is given for $r\to0+$ by 
$$C_r= 
\begin{pmatrix} 
r & 0 & 0 \\
0 & r & 0 \\
0 & 0 & 1
\end{pmatrix}$$
while the contraction $\su(2)\con\hg_3$ was thoroughly studied in \cite{Ca03}. 
\\ $\bullet$ \fbox{$\su(2)\times\RR\con\gg_T\iff\gg_T\in\{\ag_4, \hg_3\times\RR, A_{3,5}^0\times\RR, \fg_4, A_{4,9}^0\}$}
\\ Here, if we define the filiform algebra $\fg_4$ by the commutation relations 
$[e_2,e_4]=e_1$, $[e_3,e_4]=e_2$ (which is the algebra $A_{4,1}$ on \cite[pag. 20]{NP06}) 
then a contraction $\su(2)\times\RR\con\fg_4$ is given for $r\to 0+$ by 
$$C_r=
\begin{pmatrix}
-1 & 0 & 1 & 0 \\
\hfill 0 & 0 & 0 & 1 \\
\hfill 0 & 1 & 0 & 0 \\
\hfill 0 & 0 & 1 & 0 
\end{pmatrix} 
\begin{pmatrix}
r^3 & 0 & 0 & 0 \\
  0 & r^2 & 0 &0 \\
  0 & 0 & r & 0 \\
  0 & 0 & 0 & r 
\end{pmatrix}
=\begin{pmatrix}
-r^3 & 0 & r & 0 \\
\hfill 0 & 0 & 0 & r \\
\hfill 0 & r^2 & 0 & 0 \\
\hfill 0 & 0   & r & 0 
\end{pmatrix}$$
while a contraction $\su(2)\times\RR\con A_{4,9}^0$ is constructed in the Situation~\ref{Case2}({\bfi a2}) below. 
\\ ({\bfi a2}) The situation involving $\sl(2,\RR)$:   
\\ $\bullet$ \fbox{$\sl(2,\RR)\con\gg_T\iff\gg_T\in\{\ag_3, \hg_3, A_{3,4}^{-1}, A_{3,5}^0\}$}
where the Lie algebras $A_{3,4}^{-1}$ and $A_{3,5}^0$ are defined on $\RR^3$ by the commutation relations \eqref{A34-1} and \eqref{A350}, 
respectively. 
If we define $\sl(2,\RR)$ by the commutation relations 
$[e_1,e_2]=e_1,\ [e_2,e_3]=e_3,\ [e_1,e_3]=2e_2$
then 
a contraction $\sl(2,\RR)\con A_{3,4}^{-1}$ is given for $r\to0+$ by 
$$C_r= 
\begin{pmatrix} 
r & 0 & 0 \\
0 & 0 & 1 \\
0 & -1 & 0
\end{pmatrix}$$
a contraction $\sl(2,\RR)\con A_{3,5}^0$ is given for $r\to0+$ by 
$$C_r= 
\begin{pmatrix} 
0 & 0 & \frac{1}{2} \\
0 & r & 0 \\
r & 0 & \frac{1}{2}
\end{pmatrix}$$
while the contraction $\sl(2,\RR)=\su(1,1)\con\hg_3$ was thoroughly studied in \cite{Ca04}.  
\\ $\bullet$ \fbox{$\sl(2,\RR)\times\RR\con\gg_T\Leftrightarrow\gg_T\in\{\ag_4, \hg_3\times\RR, A_{3,4}^{-1}\times\RR, 
A_{3,5}^0\times\RR, \fg_4, A_{4,9}^0\}$}
\\ Here, if we define the filiform algebra $\fg_4$ by the commutation relations 
$[e_2,e_4]=e_1$, $[e_3,e_4]=e_2$ (which is the algebra $A_{4,1}$ on \cite[pag. 20]{NP06}) 
then a contraction $\sl(2,\RR)\times\RR\con\fg_4$ is given for $r\to 0+$ by 
$$C_r=
\begin{pmatrix}
 0 & 0 & \hfill 0 & 1 \\
 0 & 1 & \hfill 0 & 0 \\
 0 & 0 & -\frac{1}{2} & 0 \\
 1 & 0 & \hfill 0 & 1 
\end{pmatrix} 
\begin{pmatrix}
  r & 0 & 0 & 0 \\
  0 & r & 0 & 0 \\
  0 & 0 & r & 0 \\
  0 & 0 & 0 & 1 
\end{pmatrix}
=\begin{pmatrix}
       0 & 0 & \hfill 0 & 1\\
\hfill 0 & r & \hfill 0 & r \\
\hfill 0 & 0 & -\frac{r}{2} & 0 \\
\hfill r & 0 & \hfill 0 & 1 
\end{pmatrix}$$
while contractions $\sl(2,\RR)\times\con A_{4,9}^0$ and $\sl(2,\RR)\times\con A_{4,8}^{-1}$ 
are constructed in the Situation~\ref{Case2}(({\bfi a3})--({\bfi a4})) below.  

\item[(b)] Does (ii)$\con$$\RR^{n-2}\times$(i) hold true, if $n\ge 2$? 
 No, since the Lie algebra of (ii) is solvable while the algebras of $\RR^{n-2}\times$(i) are not; 
see \cite[Th. 1(13)]{NP06}. 
\end{itemize}
(Note that any Lie algebra of the type $\RR^k\times$(ii) is actually a Lie algebra of type (ii), 
so we need not consider separately the direct sums of~(ii) 
with abelian Lie algebras.)

\subsection{(i) vs. (iii)}\label{Case2}\hfill
\begin{itemize}
\item[(a)] Does $\RR^{k+1}\times$(i)$\con$$\RR^k\times$(iii) hold true?   
There are 4 possible situations: 
\\ ({\bfi a1}) $\RR^{k+1}\times\su(2)\con\RR^k\times A_{4,8}^{-1}$ for which the answer is negative, 
by an argument given in \cite[Rem. 12]{NP06} for $\so(3)$. 
In fact, the Killing form of $\RR^{k+3}\times\su(2)$ has 3 negative eigenvalues and the other $k$ eigenvalues are zero, 
and on the other hand the Killing form of $\RR^k\times A_{4,8}^{-1}$ has one positive eigenvalue 
and the other $k+2$ eigenvalues are zero. 
On the other hand, the number of positive eigenvalues cannot increase by a contraction process, as proved in 
\cite[Th. 1(16)]{NP06}.
\\ ({\bfi a2}) \fbox{$\RR^{k+1}\times\su(2)\con\RR^k\times A_{4,9}^0$} for which the answer is yes, by \cite[\S VIII.B]{NP06}, 
and a contraction $\su(2)\times\RR\con A_{4,9}^0$, 
is given for $r\to 0+$ by 
$$C_r
=\begin{pmatrix}
1 & 0 & 0 & 1 \\
0 & 1 & 0 & 0 \\
0 & 0 & 1 & 0 \\
0 & 0 & 0 & 1
\end{pmatrix}
\begin{pmatrix}
r^2 & 0 & 0 & 0 \\
0 & r & 0 & 0 \\
0 & 0 & r & 0 \\
0 & 0 & 0 & 1
\end{pmatrix}
=\begin{pmatrix}
r^2 & 0 & 0 & 1 \\
0 & r & 0 & 0 \\
0 & 0 & r & 0 \\
0 & 0 & 0 & 1
\end{pmatrix} $$
where $\su(2)\times\RR$ is defined on $\RR^4$ by the commutation relations 
\begin{equation}\label{su2}
[e_1,e_2]=e_3,\ [e_2,e_3]=e_1,\ [e_3,e_1]=e_2,
\end{equation} 
while $A_{4,9}^0$ is defined also on $\RR^4$ 
by the commutation relations 
\begin{equation}\label{A490}
[e_2,e_3]=e_1,\ [e_2,e_4]=-e_3,\ [e_3,e_4]=e_2. 
\end{equation}
\\ ({\bfi a3}) \fbox{$\RR^{k+1}\times\sl(2,\RR)\con\RR^k\times A_{4,8}^{-1}$}
for which the answer is yes, by \cite[\S VIII.B]{NP06}, 
and a contraction $\sl(2,\RR)\times\RR\con A_{4,8}^{-1}$, 
is given for $r\to 0+$ by 
$$C_r
=\begin{pmatrix}
0 & 1 & 0 & \hfill 0 \\
0 & 0 & 0 & \hfill 1 \\
0 & 0 & 1 & \hfill 0 \\
1 & 0 & 0 & -\frac{1}{2}
\end{pmatrix}
\begin{pmatrix}
r & 0 & 0 & 0 \\
0 & 1 & 0 & 0 \\
0 & 0 & r & 0 \\
0 & 0 & 0 & 1
\end{pmatrix}
=\begin{pmatrix}
0 & 1 & 0 & \hfill 0 \\
0 & 0 & 0 & \hfill 1 \\
0 & 0 & r & \hfill 0 \\
r & 0 & 0 & -\frac{1}{2}
\end{pmatrix} $$
where $\sl(2,\RR)\times\RR$ is defined on $\RR^4$ by the commutation relations 
\begin{equation}\label{sl2}
[e_1,e_2]=e_1,\ [e_2,e_3]=e_3,\ [e_1,e_3]=2e_2,
\end{equation} 
while $A_{4,8}^{-1}$ is defined also on $\RR^4$ 
by the commutation relations 
\begin{equation}\label{A48-1}
[e_2,e_3]=e_1,\ [e_2,e_4]=e_2,\ [e_3,e_4]=-e_3. 
\end{equation}
\\ ({\bfi a4}) \fbox{$\RR^{k+1}\times\sl(2,\RR)\con\RR^k\times A_{4,9}^0$} 
for which the answer is yes, by \cite[\S VIII.B]{NP06}, 
and a contraction $\sl(2,\RR)\times\RR\con A_{4,9}^0$, 
is given for $r\to 0+$ by 
$$
C_r
=\begin{pmatrix}
-\frac{1}{2} & \hfill 0 & \hfill \frac{1}{2} & \frac{1}{2} \\
\hfill 0 & 1 & \hfill 0 & 0 \\
-\frac{1}{2} & 0 & -\frac{1}{2} & \frac{1}{2} \\
\hfill 0 & 0 & \hfill 0 & 1
\end{pmatrix}
\begin{pmatrix}
r^2 & 0 & 0 & 0 \\
0 & r & 0 & 0 \\
0 & 0 & r & 0 \\
0 & 0 & 0 & 1
\end{pmatrix} 
=\begin{pmatrix}
-\frac{r^2}{2} & 0 & \hfill \frac{r}{2} & \frac{1}{2} \\
\hfill 0 & r & \hfill 0 & 0 \\
-\frac{r^2}{2} & 0 & -\frac{r}{2} & \frac{1}{2} \\
\hfill 0 & 0 & \hfill 0 & 1
\end{pmatrix}
$$
where $\sl(2,\RR)\times\RR$ is defined on $\RR^4$ by the commutation relations \eqref{sl2} 
while $A_{4,9}^0$ is defined also on $\RR^4$ 
by the commutation relations \eqref{A490}. 
\item[(b)] Does $\RR^k\times$(iii)$\con$$\RR^{k+1}\times$(i) hold true? 
No, since the Lie algebra of $\RR^k\times$(iii) is solvable while the algebras of $\RR^{k+1}\times$(i) are not; 
see \cite[Th. 1(13)]{NP06}. 
\end{itemize}

\subsection{(i) vs. (iv)}\label{Case3}\hfill 
\begin{itemize}
\item[(a)] Does $\RR^{k+3}\times$(i)$\con$$\RR^k\times$(iv) hold true? 
Yes, and there are 2 possible situations: 
\\ ({\bfi a1}) \fbox{$\RR^{k+3}\times\su(2)\con\RR^k\times \ng_{3,3}$} 
and a contraction $\su(2)\times\RR^3\con \ng_{3,3}$, 
is given for $r\to 0+$ by 
$$C_r
=\begin{pmatrix}
\hfill r & \hfill 0 & \hfill 0 & 0   & 0   & 0 \\
\hfill 0 & \hfill r & \hfill 0 & 0   & 0   & 0 \\
\hfill 0 & \hfill 0 & \hfill r & 0   & 0   & 0 \\
\hfill 0 & \hfill 0 &       -r & r^2 & 0   & 0 \\
      -r & \hfill 0 & \hfill 0 & 0   & r^2 & 0 \\
\hfill 0 &       -r & \hfill 0 & 0   & 0   & r^2 
\end{pmatrix}
\text{ with }
C_r^{-1}=
\begin{pmatrix}
 \frac{1}{r}  &  0                 &   0 & 0   & 0   & 0 \\
 0            &  \frac{1}{r}       &   0 & 0   & 0   & 0 \\
 0            &  0   &  \frac{1}{r}            &  0  & 0   & 0 \\
 0            &  0 & \frac{1}{r^2} & \frac{1}{r^2} & 0   & 0 \\
\frac{1}{r^2} &  0 &  0      & 0  & \frac{1}{r^2} & 0 \\
 0            &\frac{1}{r^2} &   0 & 0   & 0   & \frac{1}{r^2} \end{pmatrix}
$$
where $\su(2)\times\RR^3$ is defined on $\RR^6$ by the commutation relations~\eqref{su2}
while $\ng_{3,3}$ is defined also on $\RR^6$ 
by the commutation relations 
\begin{equation}\label{N33}
[e_1,e_2]=e_4,\ [e_2,e_3]=e_5,\ [e_3,e_1]=e_6. 
\end{equation}
\\ ({\bfi a2}) \fbox{$\RR^{k+3}\times\sl(2,\RR)\con\RR^k\times \ng_{3,3}$} 
and a contraction $\sl(2,\RR)\times\RR^3\con \ng_{3,3}$, 
is given for $r\to 0+$ by 
$$C_r
=\begin{pmatrix}
\hfill r &  0           & \hfill 0 & 0   & 0   & 0 \\
\hfill 0 &  r           & \hfill 0 & 0   & 0   & 0 \\
\hfill 0 &  0           & \hfill r & 0   & 0   & 0 \\
      -r &  0           & \hfill 0 & r^2 & 0   & 0 \\
\hfill 0 &  0           &       -r & 0   & r^2 & 0 \\
\hfill 0 &  \frac{r}{2} & \hfill 0 & 0   & 0   & r^2 
\end{pmatrix}
\text{ with }
C_r^{-1}=
\begin{pmatrix}
 \frac{1}{r}  & \hfill 0           &   0           &            0  & 0             & 0 \\
 0            & \hfill \frac{1}{r} &   0           &            0  & 0             & 0 \\
 0            & \hfill 0           & \frac{1}{r}   &            0  & 0             & 0 \\
 \frac{1}{r^2}& \hfill 0           &   0           & \frac{1}{r^2} & 0             & 0 \\
 0            & \hfill 0           & \frac{1}{r^2} &            0  & \frac{1}{r^2} & 0 \\
 0            &-\frac{1}{2r^2}     &  0            &            0  & 0             & \frac{1}{r^2} 
 \end{pmatrix}
$$
where $\sl(2,\RR)\times\RR^3$ is defined on $\RR^6$ by the commutation relations~\eqref{sl2}
while $\ng_{3,3}$ is defined also on $\RR^6$ 
by \eqref{N33}.
\item[(b)] Does $\RR^k\times$(iv)$\con$$\RR^{k+3}\times$(i) hold true? 
No, since 
the Lie algebra $\RR^k\times$(iv) is nilpotent while the algebras of $\RR^{k+3}\times$(i) are not; 
see \cite[Th. 1(14)]{NP06}. 
\end{itemize}   

\subsection{(i) vs. (v)}\label{Case4} 
\begin{itemize}
\item[(a)] Does $\RR^{k+2}\times$(i)$\con$$\RR^k\times$(v) hold true?  
Yes, and there are 2 possible situations: 
\\ ({\bfi a1}) \fbox{$\RR^{k+2}\times\su(2)\con\RR^k\times \ng_{2,1,2}$} 
and a contraction $\su(2)\times\RR^2\con \ng_{2,1,2}$, 
is given for $r\to 0+$ by 
$$C_r
=\begin{pmatrix}
\hfill r & \hfill 0 & 0   & 0   & 0 \\
\hfill 0 & \hfill r & 0   & 0   & 0 \\
\hfill 0 &        0 & r^2 & 0   & 0 \\
\hfill 0 & \hfill 1 & 0   & r^2 & 0 \\ 
      -1 & \hfill 0 & 0   & 0   & r^2 
\end{pmatrix}
\text{ with }
C_r^{-1}=
\begin{pmatrix}
 \frac{1}{r} & \hfill 0           & 0             & 0             & 0 \\
           0 & \hfill \frac{1}{r} & 0             & 0             & 0 \\
           0 & \hfill 0           & \frac{1}{r^2} & 0             & 0 \\
           0 &-\frac{1}{r^3}      & 0             & \frac{1}{r^2} & 0 \\ 
\frac{1}{r^3}& \hfill 0           & 0             & 0             & \frac{1}{r^2} 
\end{pmatrix}$$
where $\su(2)\times\RR^2$ is defined on $\RR^5$ by the commutation relations~\eqref{sl2}
while $\ng_{2,1,2}$ is defined also on $\RR^5$ 
by the commutation relations 
\begin{equation}\label{N212}
[e_1,e_2]=e_3,\ [e_1,e_3]=e_4,\ [e_2,e_3]=e_5. 
\end{equation}
\\ ({\bfi a2}) \fbox{$\RR^{k+2}\times\sl(2,\RR)\con\RR^k\times \ng_{2,1,2}$} 
and a contraction $\sl(2,\RR)\times\RR^2\con \ng_{2,1,2}$, 
is given for $r\to 0+$ by 
$$C_r
=\begin{pmatrix}
\hfill r & 0 & 0   & 0   & 0 \\
\hfill 0 & 0 & r^ 2  & 0   & 0 \\
\hfill 0 & \frac{r}{2} & 0 & 0   & 0 \\
      -1 & 0 & 0   & r^2 & 0 \\ 
\hfill 0 & 1 & 0   & 0   & r^2 
\end{pmatrix}
\text{ with }
C_r^{-1}=
\begin{pmatrix}
  \frac{1}{r} & \hfill 0 & 0   & 0   & 0 \\
            0 & \hfill 0 & \frac{2}{r}   & 0   & 0 \\
            0 & \hfill \frac{1}{r^2} & 0 & 0   & 0 \\
 \frac{1}{r^3}& \hfill 0 & 0   & \frac{1}{r^2} & 0 \\ 
            0 &0 &  -\frac{2}{r^3}  & 0   & \frac{1}{r^2} 
\end{pmatrix} $$
where $\sl(2,\RR)\times\RR^2$ is defined on $\RR^5$ by the commutation relations~\eqref{sl2}
while $\ng_{2,1,2}$ is defined also on $\RR^5$ by \eqref{N212}.
\item[(b)] Does $\RR^k\times$(v)$\con$$\RR^{k+2}\times$(i) hold true?
 No, since 
the Lie algebra $\RR^k\times$(iv) is nilpotent while the algebras of $\RR^{k+3}\times$(i) are not; 
see \cite[Th. 1(14)]{NP06}. 
\end{itemize}

\subsection{(ii) vs. (iii)}\label{Case6} 
(Recall the remark on $\RR^k\times$(ii) made in the above situation~\ref{Case1}.)
\begin{itemize}
\item[(a)] Does (ii)$\con$$\RR^{n-3}\times$(iii) hold true, for $n\ge 3$? 
No, since the derived  Lie algebra of (ii) is abelian while this is not the case for $\RR^{n-3}\times$(iii), 
hence we can use \cite[Th. 1(4)]{NP06}. 
\item[(b)] Does $\RR^{n-3}\times$(iii)$\con$(ii) hold true, for $n\ge 3$?  
There are 2 possible situations, corresponding to the two Lie algebras from (iii), 
and in order to analyze them we will need the following remarks:
\begin{align}
\label{Killing} 
\bullet\quad &  \text{the Killing form $K_{T}$ of $\gg_T$ satisfies $K_T((1,0),(1,0))=\Tr\,(T^2)$}\\
\label{unimodular}
\bullet\quad & \gg_T \emph{ is a unimodular Lie algebra }\iff \Tr\, T=0; \\
\label{center}
\bullet\quad &\Zc(\gg_T)=\Ker T; \\
\label{derived}
\bullet\quad &[\gg_T,\gg_T]=\Ran T \text{ and, more generally, }\gg_T^{j}=\Ran T^j \text{ for }j\ge1.
\end{align}
We can now study the 2 situations that can occur.  
\\ ({\bfi b1}) \fbox{$\RR^{n-3}\times A_{4,8}^{-1}\con\gg_T 
\iff \gg_T\in\{\ag_{n+1},\RR^{n-2}\times A_{3,4}^{-1}\}$ } 
\\ where $A_{3,4}^{-1}$ is the 3-dimensional Lie algebra (cf. \cite[\S VI.A]{NP06}) 
defined on~$\RR^3$ by the commutation relations 
\begin{equation}\label{A34-1}
[e_1,e_3]=e_1,\ [e_2,e_3]=-e_2. 
\end{equation}
In fact, the above implication ``$\Leftarrow$'' follows by \cite[\S VI.B]{NP06}, 
since for $n=3$ a contraction $A_{4,8}^{-1}\con A_{3,4}^{-1}\times\RR$ 
is given by 
$$C_r
=\begin{pmatrix}
0 & 0 & 0 & r \\
r & 0& 0 & 0 \\
0 & r & 0 & 0 \\
0 & 0 & 1 & 0 
\end{pmatrix}$$
while a contraction $A_{4,8}^{-1}\con \ag_4$ exists trivially. 

For the converse implication ``$\Rightarrow$'', assume there exists a contraction 
$\RR^{n-3}\times A_{4,8}^{-1}\con\gg_T$. 
By using the above remarks \eqref{center} and \eqref{derived}, 
along with \cite[Th. 1((3)--(4))]{NP06}, 
we obtain $\dim(\Ker T)\ge n-2$ and $\dim(\Ran T)\le 3$. 
It then follows that the Jordan cells in the canonical form of the linear operator $T\colon\RR^n\to\RR^n$ 
must satisfy one of the following conditions: 

\smallbreak 

$1^\circ$ We have $T=0$, and then $\gg_T=\ag_{n+1}$ (the abelian $(n+1)$-dimensional Lie algebra). 

\smallbreak 

$2^\circ$ There is precisely one nonzero cell of size $2\times 2$, and it corresponds to 
two complex conjugate eigenvalues $z,\bar z\in\CC\setminus\RR$. 
Moreover, since the Lie algebra $A_{4,8}^{-1}$ is unimodular, 
it follows that so is $\RR^{n-3}\times A_{4,8}^{-1}$, 
and then so is $\gg_T$, by \cite[Th. 1(12)]{NP06}. 
Consequently, by using the above remark~\eqref{unimodular}, 
we obtain $\Tr\,T=0$, hence $\bar z=-z$. 
Then there exists $b\in\RR\setminus\{0\}$ for which $z=\ie b$, hence 
\begin{equation}\label{Tb}
T=\begin{pmatrix}
\hfill 0  & b &             \raisebox{-3pt}{\text{\bf\huge 0}}\\
       -b & 0 &             \\
\hfill \text{\bf\huge 0}&   &     \text{\bf\huge 0}  
\end{pmatrix}
\end{equation}
and this implies $\Tr(T^2)=-2b^2<0$. 
By using \eqref{Killing}, 
we thus see that the Killing form of $\gg_T$ has one negative eigenvalue, 
the other eigenvalues being clearly equal to 0 since $\ag_n$ is an abelian ideal of $\gg_T$. 
On the other hand, the Killing form of $A_{4,8}^{-1}$ is nonnegative definite 
(see for instance \cite[\S VI.B]{NP06}), hence we cannot have $\RR^{n-3}\times A_{4,8}^{-1}\con\gg_T$, 
since \cite[Th. 1(16)]{NP06} 
shows that the number of positive eigenvalues of the Killing form cannot decrease 
by a contraction.  

\smallbreak 

$3^\circ$ There are precisely two nonzero cells of size 1 in the Jordan canonical form of $T$, 
and they correspond to some real eigenvalues $\lambda,\mu\in\RR$.
We have $\Tr\,T=0$ as in the above case $2^\circ$, hence $\mu=-\lambda\ne 0$. 
Then the Jordan canonical form of $T$ is 
\begin{equation}\label{Tlambda}
T=\begin{pmatrix}
 \lambda  & \hfill 0 &  \raisebox{-3pt}{\text{\bf\huge 0}}\\
      0  & -\lambda &             \\
 \text{\bf\huge 0}&   &     \text{\bf\huge 0}  
\end{pmatrix}
\end{equation}
and then $\gg_T$ is given by the commutation relations 
$[e_1,e_2]=\lambda e_2$ and $[e_1,e_3]=-\lambda e_3$ on $\RR^{n+1}$. 
Since $0\ne\lambda\in\RR$, these commutation relations are equivalent to \eqref{A34-1}, 
and it thus follows that $\gg_T=A_{3,4}^{-1}\times \RR^{n-2}$. 
This completes the proof of the above equivalence~({\bfi b1}). 
\\ ({\bfi b2}) \fbox{$\RR^{n-3}\times A_{4,9}^0 \con\gg_T 
\iff \gg_T\in\{\ag_{n+1},\RR^{n-2}\times A_{3,5}^0\}$ } 
\\ where $A_{3,5}^0$ is the 3-dimensional Lie algebra (cf. \cite[\S VI.A]{NP06}) 
defined on~$\RR^3$ by the commutation relations 
\begin{equation}\label{A350}
[e_1,e_3]=-e_2,\ [e_2,e_3]=e_1.
\end{equation}
In fact, the above implication ``$\Leftarrow$'' follows by \cite[\S VI.B]{NP06}, 
since for $n=3$ a contraction $A_{4,9}^0\con A_{3,5}^0\times\RR$ 
is given by 
$$C_r=\begin{pmatrix}
0 & 0 & 0 & 1 \\
1 & 0 & 0 & 0 \\
0 & 1 & 0 & 0 \\
0 & 0 & 1 & 0 
\end{pmatrix} 
\begin{pmatrix}
r & 0 & 0 & 0 \\
0 & r & 0 & 0 \\
0 & 0 & 1 & 0 \\
0 & 0 & 0 & 1 
\end{pmatrix}
=\begin{pmatrix}
0 & 0 & 0 & 1 \\
r & 0 & 0 & 0 \\
0 & r & 0 & 0 \\
0 & 0 & 1 & 0 
\end{pmatrix}$$
while a contraction $A_{4,9}^0\con \ag_4$ exists trivially. 

For the converse implication ``$\Rightarrow$'', assume there exists a contraction 
$\RR^{n-3}\times A_{4,9}^0\con\gg_T$. 
By using as above the remarks \eqref{center} and \eqref{derived}, 
along with \cite[Th. 1((3)--(4))]{NP06}, 
we obtain $\dim(\Ker T)\ge n-2$ and $\dim(\Ran T)\le 3$. 
Hence the Jordan cells in the canonical form of the linear operator $T\colon\RR^n\to\RR^n$ 
must satisfy one of the following conditions: 

\smallbreak 

$1^\circ$ We have $T=0$, and then $\gg_T=\ag_{n+1}$ (the abelian $(n+1)$-dimensional Lie algebra). 

\smallbreak 

$2^\circ$ There is precisely one nonzero cell of size $2\times 2$, and it corresponds to 
two complex conjugate eigenvalues $z,\bar z\in\CC\setminus\RR$.  
Since the Lie algebra $A_{4,9}^0$ is unimodular, 
it follows that so is $\RR^{n-3}\times A_{4,9}^0$, 
and then so is $\gg_T$, by \cite[Th. 1(12)]{NP06}. 
Consequently, by using the above remark~\eqref{unimodular}, 
we obtain $\Tr\,T=0$, hence $\bar z=-z$. 
Then there exists $b\in\RR\setminus\{0\}$ for which $z=\ie b$, hence 
we the Jordan canonical form of $T$ is \eqref{Tb} 
and then $\gg_T$ is given by the commutation relations 
$[e_1,e_2]=-b e_3$ and $[e_1,e_3]=b e_2$ on $\RR^{n+1}$. 
Since $0\ne b\in\RR$, these commutation relations are equivalent to \eqref{A350}, 
and it thus follows that $\gg_T=A_{3,5}^0\times \RR^{n-2}$.  

\smallbreak 

$3^\circ$ There are precisely two nonzero cells of size 1 in the Jordan canonical form of $T$, 
and they correspond to some real eigenvalues $\lambda,\mu\in\RR$.
We have $\Tr\,T=0$ as in the above case $2^\circ$, hence $\mu=-\lambda\ne 0$. 
Then the Jordan canonical form of $T$ is \eqref{Tlambda} 
and this implies $\Tr(T^2)=2\lambda^2>0$.
By using \eqref{Killing}, 
we thus see that the Killing form of $\gg_T$ has one positive eigenvalue, 
the other eigenvalues being equal to 0 since $\ag_n$ is an abelian ideal of $\gg_T$. 
On the other hand, the Killing form of $A_{4,9}^0$ is nonpositive definite 
(see for instance \cite[\S VI.B]{NP06}), hence we cannot have $\RR^{n-3}\times A_{4,8}^{-1}\con\gg_T$, 
since \cite[Th. 1(16)]{NP06} 
shows that the number of negative eigenvalues of the Killing form cannot decrease 
by a contraction.  
This completes the proof of the above equivalence~({\bfi b2}). 
\end{itemize}

\subsection{(ii) vs. (iv)}\label{Case7} 
(Recall the remark on $\RR^k\times$(ii) made in the above situation~\ref{Case1}.)
\begin{itemize}
\item[(a)] Does (ii)$\con$$\RR^{n-5}\times$(iv) hold true, for $n\ge 5$? 
No, since the maximum of the dimensions of the abelian ideals of a Lie algebra  
cannot decrease by a contraction process; see \cite[Th. 1(8)]{NP06}. 
\item[(b)] Does $\RR^{n-5}\times$(iv)$\con$(ii) hold true, for $n\ge 5$?  
The answer is the following: 
\\ \fbox{$\RR^{n-5}\times\ng_{3,3} \con\gg_T 
\iff \gg_T\in\{\ag_{n+1},\RR^{n-2}\times \hg_3,\RR^{n-4}\times\ng_{1,2,2}\}$ } 
\\ where $\hg_3$ is the 3-dimensional Heisenberg algebra and 
$\ng_{1,2,2}$ is the 5-dimen\-sional 2-step nilpotent Lie algebra 
defined on~$\RR^5$ by the commutation relations 
\begin{equation}\label{N122}
[e_1,e_2]=e_4,\ [e_1,e_3]=e_5.
\end{equation}
In fact, if $\RR^{n-5}\times\ng_{3,3} \con\gg_T$ then, 
by using \eqref{center}, \eqref{derived} and \cite[Th.1((3)--(4),(14))]{NP06}, 
we obtain that the linear operator $T\colon\RR^n\to\RR^n$ must 
satisfy $\dim(\Ran T)\le 3$, $\dim(\Ker T)\ge n-2$, and $T^2=0$. 
Therefore we have $2\le\dim(\Ran T)\le 3$ and $T^2=0$. 
By considering the Jordan canonical form of $T$, 
we see that there are only 3 situations that can occur:  

\smallbreak 

$1^\circ$ The canonical form of $T$ contains precisely 2 nonzero cells, 
each of them having the size $2\times 2$. 
Then the Lie algebra $\gg_T$ is defined on $\RR^{n+1}$ by the commutation relations 
$$[e_1,e_2]=e_4,\ [e_1,e_3]=e_6,$$
hence we have $\gg_T=\RR^{n-4}\times\ng_{1,2,2}$. 
For $n=5$, a contraction $\ng_{3,3} \con \RR\times\ng_{1,2,2}$ 
can be defined by 
$$C_r=\begin{pmatrix}
r & 0 & 0 & 0 & 0 & 0 \\
0 & r & 0 & 0 & 0 & 0 \\
0 & 0 & r & 0 & 0 & 0 \\
0 & 0 & 0 & r^2 & 0 & 0 \\
0 & 0 & 0 & 0 & 0 & r \\
0 & 0 & 0 & 0 & -r^2 & 0 
\end{pmatrix}$$
with respect to the above commutation relations. 

\smallbreak 

$2^\circ$ The canonical form of $T$ contains precisely 1 nonzero cell, 
 having the size $2\times 2$. 
Then the Lie algebra $\gg_T$ is defined on $\RR^{n+1}$ by the commutation relations 
$$[e_1,e_2]=e_4,$$
hence we have $\gg_T=\RR^{n-2}\times \hg_3$. 
For $n=5$, a contraction $\ng_{3,3} \con \RR^3\times\hg_3$ 
can be defined by 
$$C_r=\begin{pmatrix}
r & 0 & 0 & 0 & 0 & 0 \\
0 & r & 0 & 0 & 0 & 0 \\
0 & 0 & r & 0 & 0 & 0 \\
0 & 0 & 0 & r^2 & 0 & 0 \\
0 & 0 & 0 & 0 & r & 0 \\
0 & 0 & 0 & 0 & 0 & r 
\end{pmatrix}$$
with respect to the above commutation relations. 

\smallbreak 

$3^\circ$ We have $T=0$, and then $\gg_T=\ag_{n+1}$, and  
$\RR^{n-5}\times\ng_{3,3} \con\ag_{n+1}$ trivially. 
\end{itemize}

\subsection{(ii) vs. (v)}\label{Case8}
(Recall the remark on $\RR^k\times$(ii) made in the above situation~\ref{Case1}.)
\begin{itemize}
\item[(a)] Does (ii)$\con$$\RR^{n-4}\times$(v) hold true, for $n\ge 4$?   
No, since the maximum of the dimensions of the abelian ideals of a Lie algebra  
cannot decrease by a contraction process; see \cite[Th. 1(8)]{NP06}. 
\item[(b)] Does $\RR^{n-4}\times$(v)$\con$(ii) hold true, for $n\ge 4$? 
The answer is the following: 
\\ \fbox{$\RR^{n-4}\times\ng_{2,1,2} \con\gg_T 
\Leftrightarrow \gg_T\in\{\ag_{n+1},\RR^{n-2}\times \hg_3,\RR^{n-4}\times\ng_{1,2,2},\RR^{n-3}\times\fg_4\}$ } 
\\ where $\hg_3$ is the 3-dimensional Heisenberg algebra and 
$\fg_4$ is the 4-dimen\-sional filiform nilpotent Lie algebra 
defined on~$\RR^4$ by the commutation relations 
\begin{equation}\label{F4}
[e_1,e_2]=e_3,\ [e_1,e_3]=e_4.
\end{equation}
In fact, if $\RR^{n-5}\times\ng_{2,1,2} \con\gg_T$ then, 
by using \eqref{center}, \eqref{derived} and \cite[Th.1((3)--(4),(14))]{NP06}, 
we obtain that the linear operator $T\colon\RR^n\to\RR^n$ must 
satisfy $\dim(\Ran T)\le 3$, $\dim(\Ker T)\ge n-2$, and $T^3=0$. 
Therefore we have $2\le\dim(\Ran T)\le 3$ and $T^3=0$. 
By considering the Jordan canonical form of $T$, 
we see that there are only 4 situations that can occur: 

$1^\circ$ The canonical form of $T$ contains precisely 1 nonzero cell, 
having the size $3\times 3$. 
Then the Lie algebra $\gg_T$ is defined on $\RR^{n+1}$ by the commutation relations~\eqref{F4}
hence we have $\gg_T=\RR^{n-3}\times\fg_4$. 
For $n=4$, a contraction $\ng_{2,1,2} \con \RR\times\fg_4$ 
can be defined by 
$$C_r=\begin{pmatrix}
r & 0 & 0   & 0   & 0  \\
0 & r & 0   & 0   & 0  \\
0 & 0 & r^2 & 0   & 0  \\
0 & 0 & 0   & r^3 & 0  \\
0 & 0 & 0   & 0   & r  
\end{pmatrix}$$
with respect to the above commutation relations.

\smallbreak 

$2^\circ$ The canonical form of $T$ contains precisely 2 nonzero cells, 
each of them having the size $2\times 2$. 
Then the Lie algebra $\gg_T$ is defined on $\RR^{n+1}$ by the commutation relations 
$$[e_1,e_3]=e_4,\ [e_2,e_3]=e_5,$$
which are equivalent to \eqref{N122}, 
hence we have $\gg_T=\RR^{n-5}\times\ng_{1,2,2}$. 
For $n=5$, a contraction $\ng_{2,1,2} \con \ng_{1,2,2}$ 
can be defined by 
$$C_r=\begin{pmatrix}
1 & 0 & 0   & 0   & 0  \\
0 & r & 0   & 0   & 0  \\
0 & 0 & 1   & 0   & 0  \\
0 & 0 & 0   & 1   & 0  \\
0 & 0 & 0   & 0   & r  
\end{pmatrix}$$with respect to the above commutation relations. 

\smallbreak 

$3^\circ$ The canonical form of $T$ contains precisely 1 nonzero cell, 
 having the size $2\times 2$. 
Then the Lie algebra $\gg_T$ is defined on $\RR^{n+1}$ by the commutation relations 
$$[e_1,e_2]=e_3,$$
hence we have $\gg_T=\RR^{n-2}\times \hg_3$. 
For $n=4$, a contraction $\ng_{2,1,2} \con \RR^2\times\hg_3$ 
can be defined by 
$$C_r=\begin{pmatrix}
1 & 0 & 0   & 0   & 0  \\
0 & r & 0   & 0   & 0  \\
0 & 0 & r   & 0   & 0  \\
0 & 0 & 0   & 1   & 0  \\
0 & 0 & 0   & 0   & 1  
\end{pmatrix}$$
with respect to the above commutation relations. 

\smallbreak 

$4^\circ$ We have $T=0$, and then $\gg_T=\ag_{n+1}$, and  
$\RR^{n-4}\times\ng_{2,1,2} \con\ag_{n+1}$ trivially. 
\end{itemize}

\subsection{(iii) vs. (iii)}\label{Case8half} 
Let $\RR\ltimes_1\hg_3$ and $\RR\ltimes_2\hg_3$ be the two Lie algebras from (iii). 
\begin{itemize}
\item[(a)] Does $\RR^k\times(\RR\ltimes_1\hg_3)\con\RR^k\times(\RR\ltimes_2\hg_3)$ hold true? 
No, since 
$\RR^k\times(\RR\ltimes_1\hg_3)\not\simeq\RR^k\times(\RR\ltimes_2\hg_3)$ while the algebras of derivations 
of these two Lie algebras have the same dimension $5+k^2$ 
(see \cite[\S VI.B]{NP06} for the case $k=0$), hence we may use \cite[Th. 1(1)]{NP06}.
\item[(b)] Does $\RR^k\times(\RR\ltimes_2\hg_3)\con\RR^k\times(\RR\ltimes_1\hg_3)$ hold true? 
No, for the same reason as above. 
\end{itemize}

\subsection{(iii) vs. (iv)}\label{Case9}\hfill
\begin{itemize}
\item[(a)] Does $\RR^{k+2}\times$(iii)$\con$$\RR^k\times$(iv) hold true?  
Yes, and there are 2 possible situations: 
\\ ({\bfi a1}) \fbox{$\RR^{k+2}\times A_{4,8}^{-1}\con\RR^k\times \ng_{3,3}$} 
and a contraction $A_{4,8}^{-1}\times\RR^2\con \ng_{3,3}$, 
is given for $r\to 0+$ by 
$$C_r
=\begin{pmatrix}
\hfill r & \hfill 0           & \hfill 0           & \hfill 0 & 0             & 0 \\
\hfill 0 & \hfill \frac{1}{r} & \hfill 0           & \hfill 0 & 0             & 0 \\
\hfill 0 & \hfill 0           & \hfill \frac{1}{r} & \hfill 0 & 0             & 0 \\
\hfill 0 &       -1           & \hfill 0           & \hfill r & 0             & 0 \\
       0 & \hfill 0           & \hfill 0           & -1       & \frac{1}{r^2} & 0 \\
\hfill 0 & \hfill 0           & -1                 & \hfill 0 & 0             & r 
\end{pmatrix}
\text{ with }
C_r^{-1}=
\begin{pmatrix}
 \frac{1}{r}  &  0             & 0 & 0           & 0   & 0 \\
 0            &  r             & 0 & 0           & 0   & 0 \\
 0            &  0             & r & 0           & 0   & 0 \\
 0            &  1             & 0 & \frac{1}{r} & 0   & 0 \\
 0            &  r^2           & 0 & r           & r^2 & 0 \\
 0            &  0             & 1 & 0           & 0   & \frac{1}{r} 
 \end{pmatrix}
$$
where $A_{4,8}^{-1}\times\RR^2$ is defined on $\RR^6$ by the commutation relations 
\begin{equation}\label{A48-1bis}
[e_1,e_2]=e_2,\ [e_1,e_3]=-e_3,\ [e_2,e_3]=e_5 
\end{equation}
which are obtained from \eqref{A48-1} by relabeling the basis vectors, 
while $\ng_{3,3}$ is defined also on $\RR^6$ 
by the commutation relations \eqref{N33}. 
\\ ({\bfi a2}) \fbox{$\RR^{k+2}\times A_{4,9}^0\con\RR^k\times \ng_{3,3}$} 
and a contraction $A_{4,9}^0\times\RR^2\con \ng_{3,3}$, 
is given for $r\to 0+$ by 
$$C_r
=\begin{pmatrix}
\hfill r & \hfill 0 & \hfill 0 & 0   & 0   & 0 \\
\hfill 0 & \hfill r & \hfill 0 & 0   & 0   & 0 \\
\hfill 0 & \hfill 0 & \hfill r & 0   & 0   & 0 \\
       0 & \hfill 0 & -1       & r   & 0   & 0 \\
\hfill 0 & \hfill 0 & \hfill 0 & 0   & r^2 & 0 \\
\hfill 0 &  -1      & \hfill 0 & 0   & 0   & r 
\end{pmatrix}
\text{ with }
C_r^{-1}=
\begin{pmatrix}
 \frac{1}{r}  & \hfill 0           &   0           &            0  & 0             & 0 \\
 0            & \hfill \frac{1}{r} &   0           &            0  & 0             & 0 \\
 0            & \hfill 0           & \frac{1}{r}   &            0  & 0             & 0 \\
 0            & \hfill 0           & \frac{1}{r^2} & \frac{1}{r}   & 0             & 0 \\
 0            & \hfill 0           &  0            &            0  & \frac{1}{r^2} & 0 \\
 0            & \frac{1}{ r^2}     &  0            &            0  & 0             & \frac{1}{r} 
 \end{pmatrix}
$$
where $A_{4,9}^0\times\RR^2$ is defined on $\RR^6$ by the commutation relations 
\begin{equation}\label{A490bis}
[e_1,e_2]=e_3,\ [e_1,e_3]=-e_2,\ [e_2,e_3]=e_5 
\end{equation}
which are obtained from \eqref{A490} by relabeling the basis vectors,
while $\ng_{3,3}$ is defined also on $\RR^6$ 
by \eqref{N33}.
\item[(b)] Does $\RR^k\times$(iv)$\con$$\RR^{k+2}\times$(iii) hold true?   
No, since $\RR^k\times$(iv) is nilpotent while $\RR^{k+2}\times$(iii) is not, 
hence we can use \cite[Th. 1(14)]{NP06}. 
\end{itemize}

\subsection{(iii) vs. (v)}\label{Case10} 
This is similar to situation~\ref{Case9}. 
\begin{itemize}
\item[(a)] Does $\RR^{k+1}\times$(iii)$\con$$\RR^k\times$(v) hold true?   
Yes, and there are 2 possible situations: 
\\ ({\bfi a1}) \fbox{$\RR^{k+1}\times A_{4,8}^{-1}\con\RR^k\times \ng_{2,1,2}$} 
and a contraction $A_{4,8}^{-1}\times\RR\con \ng_{2,1,2}$, 
is given for $r\to 0+$ by 
$$C_r
=\begin{pmatrix}
\hfill r & \hfill 0           & \hfill 0           & \hfill 0   & 0              \\
\hfill 0 & \hfill r           & \hfill 0           & \hfill 0   & 0              \\
\hfill 0 & \hfill 0           & \hfill r^2         & \hfill 0   & 0              \\
\hfill 0 &       -1           & \hfill 0           & \hfill r^2 & 0              \\
       0 & \hfill 0           & \hfill 0           &  0         & r^3  
\end{pmatrix}
\text{ with }
C_r^{-1}=
\begin{pmatrix}
 \frac{1}{r}  &  0             & 0             & 0             & 0    \\
 0            &  \frac{1}{r}   & 0             & 0             & 0    \\
 0            &  0             & \frac{1}{r^2} & 0             & 0    \\
 0            &  \frac{1}{r^3} & 0             & \frac{1}{r^2} & 0    \\
 0            &  0             & 0             & 0             & \frac{1}{r^3}   
 \end{pmatrix}
$$
where $A_{4,8}^{-1}\times\RR$ is defined on $\RR^5$ by the commutation relations 
\begin{equation*}
[e_1,e_2]=e_3,\ [e_1,e_3]=e_2,\ [e_2,e_3]=e_5 
\end{equation*}
which are obtained by writing \eqref{A48-1bis} 
with respect to the basis 
$$\{e_1,e_2+e_3,e_2-e_3,-2e_5\}$$ 
while $\ng_{2,1,2}$ is defined also on $\RR^5$ 
by the commutation relations \eqref{N212}. 
\\ ({\bfi a2}) \fbox{$\RR^{k+1}\times A_{4,9}^0\con\RR^k\times \ng_{2,1,2}$} 
and a contraction $A_{4,9}^0\times\RR\con \ng_{2,1,2}$, 
is given for $r\to 0+$ by 
$$C_r
=\begin{pmatrix}
\hfill r & \hfill 0           & \hfill 0           & \hfill 0   & 0              \\
\hfill 0 & \hfill r           & \hfill 0           & \hfill 0   & 0              \\
\hfill 0 & \hfill 0           & \hfill r^2         & \hfill 0   & 0              \\
\hfill 0 &        1           & \hfill 0           & \hfill r^2 & 0              \\
       0 & \hfill 0           & \hfill 0           &  0         & r^3  
\end{pmatrix}
\text{ with }
C_r^{-1}=
\begin{pmatrix}
 \frac{1}{r}  & \hfill 0             & 0             & 0             & 0    \\
 0            & \hfill \frac{1}{r}   & 0             & 0             & 0    \\
 0            & \hfill 0             & \frac{1}{r^2} & 0             & 0    \\
 0            &  -\frac{1}{r^3}      & 0             & \frac{1}{r^2} & 0    \\
 0            & \hfill 0             & 0             & 0             & \frac{1}{r^3}   
 \end{pmatrix}
$$where $A_{4,9}^0\times\RR$ is defined on $\RR^5$ by 
the commutation relations \eqref{A490bis},
while $\ng_{2,1,2}$ is defined also on $\RR^5$ 
by \eqref{N212}.
\item[(b)] Does $\RR^k\times$(v)$\con$$\RR^{k+1}\times$(iii) hold true?  
No, since $\RR^k\times$(v) is nilpotent while $\RR^{k+1}\times$(iii) is not, 
hence we can use \cite[Th. 1(14)]{NP06}. 
\end{itemize}

\subsection{(iv) vs. (v)}\label{Case11}\hfill 
\begin{itemize}
\item[(a)] Does $\RR^k\times$(iv)$\con$$\RR^{k+1}\times$(v) hold true? 
No, since $\RR^k\times$(iv) is 2-step nilpotent while $\RR^{k+1}\times$(v) is 3-step nilpotent; see \cite[Th. 1(14)]{NP06}.
\item[(b)] Does $\RR^{k+1}\times$(v)$\con$$\RR^k\times$(iv) hold true? 
No, since the dimension of the derived algebra cannot increase by a contraction process; see \cite[Th. 1(5)]{NP06}.
\end{itemize}

\section{Contractions of Lie algebras with hyperplane abelian ideals}\label{Case5}

In this section we discuss the case (ii) vs. (ii), 
that is, contractions within the Lie algebras of type (ii) from Problem~\ref{probl1}. 
In particular we provide a proof for Theorem~\ref{main}\eqref{main_item3}; 
see Proposition~\ref{24dec2013} below. 

There are many semidirect products $\RR\ltimes\ag_n$, 
determined by the various linear operators on $\ag_n=\RR^n$, 
and we are asking here about the contractions between the various Lie algebras  
obtained in this way. 
(Recall the remark on $\RR^k\times$(ii) made in the above situation~\ref{Case1}.)
The above Proposition~\ref{prop1} belongs to this circle of ideas, 
but it does not provide the complete answer. 
We establish below a few more results of this type and settle the question completely. 
So we wish to find necessary/sufficient conditions on $T,T_0\colon\RR^n\to\RR^n$ 
ensuring that $\gg_T\con\gg_{T_0}$.  
It will be convenient to use the notation $\Sc(T):=\{CTC^{-1}\mid C\in \GL(n,\RR)\}$ 
for the similarity orbit,  
$\overline{\Sc(T)}$ for the closure of $\Sc(T)$,  
and $\Ac(T):=\{T_0\in M_n(\RR)\mid \gg_T\con\gg_{T_0}\}$, 
for any linear operator identified to a matrix $T\in M_n(\RR)$. 

We denote by $C(T)$ the double cone generated by $S(T)$, that is, 
$$C(T):=\bigcup_{\lambda \in \RR}\lambda \Sc(T).$$

\begin{proposition}\label{24dec2013} 
Let $T, S\not= 0$. 
\begin{enumerate} 
\item If ${\mathfrak g}_T\simeq {\mathfrak g}_S$ then $S\in C(T)\setminus (0)$.

\item If ${\mathfrak g}_T\con {\mathfrak g}_S$ then $S\in \overline {C(T)\setminus (0)}$.

\end{enumerate}  
\end{proposition}

\begin{proof} (1) Let $D: {\mathfrak g}_T\rightarrow {\mathfrak g}_S$ be a Lie algebra isomorphism.
We write $D$ as the block matrice 
$$D= 
\begin{pmatrix} 
a& b^t \\
c & A  \\
\end{pmatrix}$$
where $a\in \RR$, $b,c\in M_{n1}(\RR)$ and $A\in M_n(\RR)$.

Note that for each $u\in M_{n1}(\RR)$, the map 
$$E(u)=\begin{pmatrix} 
1& 0 \\
u & I_n  \\
\end{pmatrix}$$
is a Lie algebra isomorphism from ${\mathfrak g}_T$ to ${\mathfrak g}_S$.

First, we assume that $A$ is invertible. Then the map
$$D_1:=DE(-A^{-1}c)=\begin{pmatrix} 
a-b^tA^{-1}c& b^t \\
0 & A \\
\end{pmatrix}$$
is also a Lie algebra isomorphism from ${\mathfrak g}_T$ to ${\mathfrak g}_S$. In particular,
one has $a':=a-b^tA^{-1}c\not= 0$. By writing the equality
$$[D_1(1,0), D_1(0,v)]_S=D_1([(1,0),(0,v)]_T),$$
we obtain that $a'SA=AT$ hence $S\in C(T)\setminus (0)$.

Now assume that $A$ is not invertible. Since $D$ is invertible, we have necessarily that $b\not= 0$. 
We choose $u\notin \Ran A$ and we consider the map
$$D_2:=E(u)D=\begin{pmatrix} 
a& b^t \\
au+c & ub^t+A \\
\end{pmatrix}.$$
Then $D_2$ is a Lie algebra isomorphism from ${\mathfrak g}_T$ to ${\mathfrak g}_S$ and we claim
that $A':=ub^t+A$ is invertible. Indeed, if $x\in \Ker A$ then one has $Ax=-(b^tx)u$. Since $u\notin \Ran A$,
we have $b^tx=0$ and we get $D(0,x)=(b^tx,Ax)=(0,0)$ hence $x=0$. Thus we have reduced this case to
the preceding one.

(2) Let $(C_r)$ be a contraction from ${\mathfrak g}_T$ to ${\mathfrak g}_S$. As before, we can write
$$C_r=\begin{pmatrix} 
a_r& b_r^t \\
c_r & A_r  \\
\end{pmatrix}$$
where $a_r\in \RR$, $b_r,c_r\in M_{n1}(\RR)$ and $A_r\in M_n(\RR)$.

First, we assume that $A_r$ is invertible for each $r$. Then 
$$C_r':=DE(-A_r^{-1}c_r)=\begin{pmatrix} 
a_r-b_r^tA_r^{-1}c_r& b_r^t \\
0 & A_r \\
\end{pmatrix}$$
is also a contraction from ${\mathfrak g}_T$ to ${\mathfrak g}_S$ and one has
$a_r':=a_r-b_r^tA_r^{-1}c_r\not= 0$. Note that
$$C_r'^{-1}=\begin{pmatrix} 
a_r'^{-1}& -a_r'^{-1}b_r^tA_r^{-1} \\
0 & A_r^{-1} \\
\end{pmatrix}.$$
Then, by writing the equality
$$\lim_{r\to 0}C_r'^{-1}[C_r'(1,0),C_r'(0,v)]_T=[(1,0),(0,v)]_S$$
we immediately obtain  $\lim\limits_{r\to 0}a_r'A_r^{-1}TA_rv=Sv$ hence  $S\in \overline {C(T)\setminus (0)}$.

Finally, if the $A_r$ are not necessarily invertible then we can find a function $\varepsilon_r$ such that $C_r+\varepsilon_rI_{n+1}$
is also a contraction from ${\mathfrak g}_T$ to ${\mathfrak g}_S$ and that $A_r+\varepsilon_rI_n$ is invertible. Then this case
reduces to the preceding one.
\end{proof}

\begin{remark}
\normalfont
We note some facts that complement the above general result. 

\begin{itemize}
\item[{\bf (I)}]\label{I} It follows by Proposition~\ref{24dec2013} or directly 
that if $T_0\in\overline{\Sc(T)}$, then $\gg_T\con\gg_{T_0}$. 
Equivalently, $\overline{\Sc(T)}\subseteq\Ac(T)$.  
One finds in \cite[Th. 1.1]{BH79} a description of $\overline{\Sc(T)}$ 
in terms of the minimal polynomial of $T$ if all the eigenvalues of $T$ are real. 
\\ \emph{When do we have $\overline{\Sc(T)}=\Ac(T)$}?  
This is the case at least when $T$ is nilpotent; see {\bf (III)}. 
  
\item[{\bf (II)}] If $T^m=0$ for some $m\ge 1$ and $\gg_T\con\gg_{T_0}$, then also $T_0^m=0$. 
\\ This follows by \cite[Th. 1(14)]{NP06}, since $T^m=0$ if and only if 
$\gg_T$ is an $(m+1)$-step nilpotent Lie algebra, and similarly for $T_0$. 

\item[{\bf (III)}] If the operator $T$ is nilpotent, say $T^m=0$, then  
$$\gg_T\con\gg_{T_0}\iff(\forall j\in\{1,\dots,m\})\quad \rank T_0^j\le\rank T^j.$$
If the right-hand side of the above equivalence holds true, 
then \cite[Prop. 3.1]{BH79} (which holds true over $\RR$ as well) 
ensures that $T_0\in\overline{\Sc(T)}$, hence we may use {\bf (I)} above. 

Conversely, if $\gg_T\con\gg_{T_0}$, then \cite[Th. 1(4)]{NP06} implies that 
for every $j\ge 1$ we have $\dim\gg_{T_0}^{(j)}\le\dim\gg_T^{(j)}$, 
and then the conclusion follows by since $\gg_{T_0}^{(j)}=\Ran T_0^j$ 
and  $\gg_T^{(j)}=\Ran T^j$ (see \eqref{derived}). 

\item[{\bf (IV)}] If $T^n=0\ne T^{n-1}$ (that is, $T$ is a Jordan cell of size~$n$), 
then for every nilpotent $T_0$ we have $\gg_T\con\gg_{T_0}$. 
\\ This follows by {\bf (III)}, but it can be proved directly as follows. 
For every $d_1,\dots,d_n\in\RR\setminus\{0\}$ let $\diag(d_1,\dots,d_n)$ 
denote the diagonal matrix with these diagonal entries. 
Then it is easily checked that 
$$\diag(d_1,\dots,d_n)T\diag(d_1,\dots,d_n)^{-1}=
\begin{pmatrix}
0 & \frac{d_1}{d_2} &        & \raisebox{-5pt}{\text{\bf\huge 0}}\\
  &\ddots                &\ddots  & \\
  &                      &   \ddots     & \frac{d_{n-1}}{{d_n}} \\
\text{\bf\huge 0}  &                      &        & 0
\end{pmatrix}
 $$ 
where we wrote $T$ as an upper triangular Jordan cell. 
Now note that in the Jordan canonical form of the nilpotent operator $T_0$ 
 the diagonal situated just above the main diagonal is a sequence 
of $n-1$ entries that take the values  0 or 1, 
and $T_0$ is uniquely determined by the positions $j_1<\cdots<j_q$ of the entries 
equal to 0 in that sequence. 
If we now pick any integers $k_1\le\cdots\le k_n$ such that 
$k_j<k_{j+1}$ if and only if $j\in\{j_1,\dots,j_q\}$, 
then it follows by the above matrix computation 
that 
$$\lim\limits_{r\to 0}\diag(r^{k_1},\dots,r^{k_n})T\diag(r^{k_1},\dots,r^{k_n})=T_0$$ 
hence $T_0\in\overline{\Sc(T)}$, and then $\gg_T\con\gg_{T_0}$ by {\bf (I)}. 

\end{itemize}
\end{remark}







\end{document}